\documentclass{article}
\usepackage{amsmath,amsthm,amssymb,amsfonts}
\usepackage[english]{babel}
\usepackage{geometry}
\usepackage{verbatim}
\usepackage{booktabs}

\makeatletter
\newcommand{\rmnum}[1]{\romannumeral #1}
\newcommand{\Rmnum}[1]{\expandafter\@slowromancap\romannumeral #1@}
\makeatother

\newtheorem{The}{Theorem}
\newtheorem{Lem}{Lemma}
\newtheorem{Def}{Definition}
\newtheorem{Rem}{Remark}
\newtheorem{Con}{Conjecture}

\newtheorem{Exa}{Example}
\newtheorem{Prop}{Proposition}
\newtheorem*{Mthe1}{Theorem 1}
\newtheorem*{Mthe2}{Theorem 4}
\newtheorem*{Mthe3}{Theorem 7}
\newtheorem*{Stru}{Structure of the paper}

\title{Unimodular Smooth Fano Polytopes and their Relation with Ewald Conditions}

\author{Binnan Tu\\ Graduate School of Information Science and Technology, \\ Osaka University}
\date{}
\begin{document}
\maketitle

\begin{abstract}
Smooth Fano polytopes (SFPs) play an important role in toric geometry and combinatorics. In this paper, we introduce a specific subcollection of them, i.e., the unimodular smooth Fano polytopes (USFPs). In Section 2, they are verified to satisfy the three (weak, strong, star) Ewald conditions. Besides, a characterisation of USFPs is provided as a corollary of the famous Seymour's decomposition theorem. Then, we briefly introduce the works by Luis Crespo on deeply monotone polytopes and give a proof of the claim that any deeply monotone polytope is in fact the dual polytope of some USFP. In other words, we extend his results on deeply monotone polytopes to the case of USFPs.
\end{abstract}

\section{Introduction}
Smooth Fano polytopes have been intensively studied for several decades. Their unimodular equivalence classes are convex geometric objects corresponding to the ones of smooth toric Fano varieties. Thus, it allows one to pose, consider and solve problems about them from both combinatorial and algebraic view points. The dual concept of smooth Fano polytopes is known as the monotone (moment) polytopes, which correspond to monotone symplectic manifolds. For more details about them on symplectic geometry, we refer to the book by Cannas da Silva \cite{CdS}. In addition, they also have an emphasis in classical mechanics. See, e.g., \cite{AM}. Up to unimodular equivalence, smooth Fano polytopes and monotone polytopes share the same amount for any fixed dimension. Currently, they have been completely classified until dimension $9$ \cite{LP}, \cite{Obro}. However, it's hard to proceed the process in higher dimension since the number increases rapidly. Meanwhile, people are also interested in seeking for possible bounds on invariants. For example, Casagrande showed that the maximal number of vertices of a smooth Fano $n$-polytope is $3n$ \cite{Cas}. Besides, there are also bounds on volume and lattce points. An interesting example without smoothness can be found in \cite{Nill2}, Theorem $B$ and $C$. For the smooth case, an attractive open question is Ewald conjecture \cite{Ewa}, which claims a volumn bound for any smooth Fano $n$-polytope. A most recent partial result to this conjecture is given by Crespo \cite{CPS}. 
\newline
To finally prove Ewald conjecture, we are motivated to generalize his result to a larger subcollection of smooth Fano polytopes. So, the goal of this paper is to to show that the set of so-called unimodular smooth Fano polytopes is such an extension.

\subsection{Smooth Fano polytopes and monotone polytopes}
\begin{Def}
\textup{Let $P \subset N_R (\cong \mathbb{R}^n)$ be an $n$-dimensional lattice polytope, then it is said to be}
\begin{itemize}
\item[(1)] projective \emph{if it contains the origin as its interior;}
\item[(2)] Fano \emph{if it's projective and any of its vertices is a primitive lattice point;}
\item[(3)] reflexive \emph{if its dual polytope is still a lattice polytope. The dual polytope $P^*$ consists of $x \in \mathbb{R}^n$ s.t. $\langle x,y \rangle \geq -1$ for all $y \in P$, where $\langle x,y \rangle$ is the standard inner product of $\mathbb{R}^n$.}
\end{itemize}
\end{Def}
\noindent
Nowadays, the two `dual' definitions of smooth lattice polytopes are both commonly used by mathematicians. One starts with the polyhedral cone, whose apex is the origin, over a facet of the polytope, e.g., \cite{Obro}. The other regards the corner at a fixed vertex of the polytope as a cone, e.g. \cite{CPS},\cite{McDuff}. Let $N \cong \mathbb{Z}^n$ be a lattice with associated real vector space  $N_R \cong N \otimes_{\mathbb{Z}} \mathbb{R}$ and let $M$ be the dual lattice of $N$ with associated real vector spaces $M_R \cong M \otimes_{\mathbb{Z}} \mathbb{R}$. Then we have the two definitions as follows:

\begin{Def}
\textup{Let $P \subset N_R (\cong \mathbb{R}^n)$ be an $n$-dimensional lattice polytope, then it is said to be}
\begin{itemize}
\item[(4a)] simplicial \emph{if there are precisely $n$ vertices for each facet;}
\item[(5a)] smooth \emph{if it is simplicial and the vertices of any facet of $P$ form a basis of the lattice $N$.} 
\end{itemize}
\end{Def}

\begin{Def}
\textup{Let $Q \subset M_R (\cong \mathbb{R}^n)$ be an $n$-dimensional lattice polytope, then it is said to be}
\begin{itemize}
\item[(4b)] simple \emph{if there are precisely $n$ edges meeting at each fixed vertex;}
\item[(5b)] smooth \emph{if it is simple and the primitive edge-direction vectors at each vertex form a basis of the lattice $M$.}
\end{itemize}
\end{Def}
\noindent
When talking about \emph{smooth Fano polytope}, we always mean the polytope in $N_R$, i.e. Definition $2$. Besides, the smooth reflexive polytope in $M_R$ with Definition $3$ is also called a \emph{monotone polytope}, whose name arises from the research of sympletic toric geometry. In most cases where the definition is clear, we simply use $\mathbb{R}^n$ as our vector space for polytopes. 
\\
\\
\textbf{Unimodular equivalence}
\par
Let $P$ and $P'$ be two lattice polytopes in $\mathbb{R}^n$, then we say that $P$ and $P'$ are \emph{unimodularly equivalent} if there exists a unimodular transformation which maps $P$ to $P'$. Here, a \emph{unimodular transformation} is a linear map whose representing matrix is an integer matrix and has determinant $\pm 1$. All the polytoeps in this paper, if nothing more is mentioned, are up to unimodular equivalence.

\subsection{Ewald conjectures and Ewald conditions}
In the research of smooth Fano polytopes, there is a long standing open question raised by G\"{u}nter Ewald in 1988 \cite{Ewa}. He conjectured that any $n$-dimensional smooth Fano polytope can be embedded into ${[-1,1]}^n$. After that, Mikkel $\O$bro \cite{Obro} mentioned a stronger version of Ewald conjecture and gave a positive answer to the case up to dimension $7$. Since it's hard to prove these Ewald conjectures directly, a natural idea is to find a possible collection of polytopes with more restrictions such that they satisfy the condition of these conjectures, and then `expand' it to the original conjectures. In other words, we are trying to seak for the maximal subset of smooth Fano polytopes that are true for the so-called weak and strong `Ewald conditions'. Moreover, Dusa McDuff introduced another Ewald condition called star Ewald, inspired by problems in toric symplectic geometry \cite{McDuff}. The concrete expression of these conditions are listed below.

\begin{Def}
\textup{Let $P$ be an $n$-dimensional projective polytope, then we say that}
\begin{itemize}
\item \emph{it satisfies} the weak Ewald condition \emph{if $P$ can be embedded into the hypercube ${[-1,1]}^n$ via a unimodular transformation;}
\item \emph{equivalently, it satisfies} the weak Ewald condition \emph{if the symmetric point set $\mathcal{E}(P^*)$ of its dual polytope $P^*$ contains a unimodular basis of $\mathbb{Z}^n$, where the \textit{symmetric point set} is defined to be $$\mathcal{E}(P^*) := P^* \cap (-P^*) \cap \mathbb{Z}^n ;$$ }
\item \emph{it satisfies} the strong Ewald condition \emph{if for any vertex $v \in P$, there exists a unimodular transformation $\phi$, s.t. $\phi(P) \subset {[-1,1]}^n$ and $\phi(v)=\sum_{i=1}^{n} e_i$, where $e_i$ are standard basis vectors of $\mathbb{R}^n$;}
\item \emph{equivalently, it satisfies} the strong Ewald condition \emph{if for any facet $F$ of its dual polytope $P^*$, the set $\mathcal{E}(P^*) \cap F$ contains a unimodular basis of $\mathbb{Z}^n$;}
\item \emph{$P^*$ satisfies} the star Ewald condition \emph{if every face of $P^*$ satisfies the star Ewald condition. A face $f$ of $P^*$ is said to be} star Ewald \emph{if there exists a symmetric point $\lambda \in \mathcal{E}(P^*)$ s.t. $\lambda \in \mathrm{Star}^{*}(f)$ and $-\lambda \notin \mathrm{Star}(f)$. Here, $\mathrm{Star}(f)$ is the union of all facets (face of codimension $1$) containing face $f$, and $\mathrm{star}(f)$ is the union of all ridges (face of codimension $2$) containing face $f$. Besides, $\mathrm{Star}^*(f):= \mathrm{Star}(f) \backslash \mathrm{star}(f)$.}

\end{itemize}
\end{Def}

\noindent
Furthermore, Benjamin Nill generalized the weak Ewald conjecture in 2009 to the case where reflexivity of the dual polytope is not required \cite{Nill}:
\begin{Con}[Generalized Ewald conjecture \cite{Nill}] 
Let $P \subset N_\mathbb{R}$ be an $n$-dimensional polytope whose dual polytope $P^*$ is a smooth projective lattice polytope in $M_\mathbb{R}$, then $P$ lies in the hypercube ${[-1,1]}^n$.
\end{Con}
A recent result in dimension $2$ and some partial results in dimension $3$ have been verified by Luis Crespo \cite{CPS}. Though it seems more general without the requirement of reflexivity, Nill's conjecture is believed to be implied by the weak Ewald conjecture \cite{CPS}.

\subsection{Main Theorems}
There are three main theorems of this paper:
\begin{Mthe1}
Let $P$ be an $n$-dimensional unimodular smooth Fano polytope with $m$ vertices, then
\begin{itemize}
\item[(\rmnum{1})] $P$ satisfies the strong Ewald condition, and
\item[(\rmnum{2})] the dual polytope $P^*$ of $P$ satisfies the star Ewald condition.
\end{itemize}
\end{Mthe1}

\begin{Mthe2}
Let $P$ be a unimodular smooth Fano polytope. Then its coordinate matrix $M$ can be constructed from graphic matrices by elementary operations, dualizing and $k$-sums for $k=1,2,3$. 
\end{Mthe2}

\begin{Mthe3}
Any deeply monotone polytope is the dual polytope of a unimodular smooth Fano polytope.
\end{Mthe3}

\begin{Stru}
\textup{The contents of this paper are mainly divided into two parts:}
\begin{itemize}
\item[1.] \textup{In Section $2$, we firstly introduce the unimodular smooth Fano polytopes (USFPs, for short). Then, we discuss their relation with weak, strong and star Ewald conditions, where the proof of Theorem $1$ is given. After that, we recall some backgrounds on matroid theory and use Seymour's decomposition theorem to induce the characterisation of USFPs. }
\item[2.] \textup{In Section $3$, we briefly recall the works of Luis Crespo on deeply monotone polytopes (DMPs, for short). Then, in Subsection $3.2$, The relation between DMPs and USFPs is clarified by Theorem $7$. And at the end of this paper, we list the numbers of several types of polytopes up to dimension $5$ and propose a natural conjecture of UT-free monotone polytopes and USFPs by observation.}
\end{itemize}
\end{Stru}

\section{Unimodular Smooth Fano Polytopes}
Let $P \subset \mathbb{R}^n$ be a lattice polytope, then the (coordinate) matrix of $P$ is an $m \times n$-matrix $(m>n)$ whose each row is the coordinates of a vertex. The polytope $P$ is said to be \textit{unimodular} if any $n$ vertices form either a basis of the lattice or are linearly dependent. In other words, any $n \times n$-submatrix has determinant $\pm 1$ or $0$. 
\\
\\
\noindent
\textbf{Standard form.} There is a natural form of the coordinate matrix of a USFP (an abbreviation for unimodular smooth Fano polytope), which is unique up to unimodular equivalence. Concretely speaking, let $M$ be the $m \times n$-matrix ($m >n$) of a USFP and choose any facet of this polytope. Denote its matrix as $N$ and the rest part in the matrix of the polytope as $R$, then via a unimodular transformation $T$, we have:
$$
T: 
M=
\begin{pmatrix}
R
\\ 
N
\end{pmatrix}
\longrightarrow
MN^{-1}=
\begin{pmatrix}
RN^{-1}
\\ 
I_{n\times n}
\end{pmatrix}
$$ 
where $M\cdot N^{-1}$ is called the \textit{standard form} of the matrix of the USFP.
\\
\\

A matrix is said to be \textit{totally unimodular} if any of its square submatrices has determinant $\pm 1$ or $0$. A smooth Fano polytope is said to be \textit{totally unimodular} if its (coordinate) matrix is totally unimodular. 
\begin{Prop}
The standard form of the matrix of a USFP is totally unimodular. 
\end{Prop}

\begin{proof}
Let $P$ be an $n$-dimensional USFP and $M$ be its matrix. Then all the $n\times n$-submatrices of $M$ has determinant $\pm 1$ or $0$. Take one of those submatrices whose determinant is $1$ and transform it into the identity matrix $I_{n\times n}$ via elementary column operations (i.e., interchange of two columns, multiplication $-1$ to a column or subtracting one column by another column). The result matrix $M'$ is a standard form and can be written as:
$$
\begin{pmatrix}
m_{11} & m_{12} & m_{13} & \cdots & m_{1n} \\
m_{21} & m_{22} & m_{23} & \cdots & m_{2n} \\
\vdots & \vdots & \vdots & \vdots & \vdots \\
m_{m1} & m_{m2} & m_{m3} & \cdots & m_{mn} \\
1 & 0 & 0 & \cdots & 0 \\
0 & 1 & 0 & \cdots & 0 \\
0 & 0 & 1 & \cdots & 0 \\
\vdots & \vdots & \vdots & \vdots & \vdots \\
0 & 0 & 0 & \cdots & 1 
\end{pmatrix}
$$
Since elementary column operations do not change the determinants of all the $n\times n$-submatrices of $M$, any $n\times n$-submatrices of $M'$ again has determinant $\pm 1$ or $0$. Let $N$ be a $t\times t$-submatrix of $M'$, where $t$ is the cardinality of set $I$ (or $J$), and $I\subset \{1,\dots,m+n\}$ and $J\subset \{1,\dots, n\}$ are row-indices set and column-indices set of $N$, respectively. Then, we can construct an $n\times n$-submatrix $L$ of $M'$ which consists of: 
\begin{itemize}
\item $i$-th row of $M'$, for any $i \in I$;
\item $(m+j)$-th row of $M'$, for any $j \in \{1,\dots,n\} \backslash J$.
\end{itemize}
By the construction above, having $i$-th rows of $M'$ included makes sure that $N$ is a submatrix of $L$ and having $(m+j)$-th rows of $M'$ included provides an identity $(n-t)\times (n-t)$-submatrix. Then, by rearranging rows and columns, we have:
$$
L=
\begin{pmatrix}
N & * \\
0 & Id_{(n-t)\times(n-t)}    
\end{pmatrix}
$$ 
Because of unimodularity, $det(L) = \pm 1$ or $0$. Then, $det(L)= det(N) \cdot det({Id}_{{(n-t)}\times {(n-t)}})= \pm 1$ or $0$. Therefore, $det(N)=\pm 1$ or $0$. Since $N$ is an arbitrary submatrix of $M'$, $M'$ is totally unimodular.

\end{proof}
On the other hand, totally unimodular SFPs are trivially USFPs. Thus, for an SFP, `unimodular' is equivalent to `totally unimodular'.

\subsection{Relation with three Ewald conditions}

From the standard form of USFPs, any of their vertices has to lie in ${[-1,1]}^n$. So it obviously satisfies the weak Ewald condition. In the following proposition, we will see that they are also correct for strong and star Ewald conditions. Before that, two lemmas have to be introduced first.

\begin{Lem}[\cite{AS}, Theorem 19.5]
Let $M$ be a totally unimodular matrix with full-column rank. If $N$ is a square submatrix of $M$ with $det(N)=\pm 1$, then $M\cdot N^{-1}$ is also totally unimodular.
\end{Lem}

The next lemma is a characterisation of totally unimodular matrices
\begin{Lem}[\cite{AS}, Theorem 19.3 (iv)]
Let $M$ be a totally unimodular matrix. Then each collection of rows of $M$ can be split into two parts so that the sum of the rows in one part minus the sum of the rows in the other part is a vector with entries only $0$, $1$ and $-1$.
\end{Lem}

\begin{Exa}
Consider the rows of the $6\times 4$-totally unimodular matrix below. The sum of the first $4$ rows minus the sum of the rest $2$ rows has entries only $\pm 1$.
$$
\begin{pmatrix}
1 & 1& 1& 1 & -1& -1
\end{pmatrix}
\cdot
\begin{pmatrix}
-1 & 1& 0& 1\\
1 & 0& 1& -1\\
0 & -1& 0& 0\\
1 & -1& 0& 0\\
0 & 0& 1& -1\\
0 & 0& -1& 0
\end{pmatrix}
=
\begin{pmatrix}
1 & -1& 1& 1
\end{pmatrix}
$$
\end{Exa}

The following theorem is one of the main results of this paper:
\begin{The}
Let $P$ be an $n$-dimensional USFP with $m$ vertices, then
\begin{itemize}
\item[(\rmnum{1})] $P$ satisfies the strong Ewald condition, and
\item[(\rmnum{2})] the dual polytope $P^*$ of $P$ satisfies the star Ewald condition.
\end{itemize}
\end{The}

\begin{proof} Let $M$ be the standard form of the (coordinate) matrix of $P$. Then $M$ is totally unimodular by Lemma $1$.
\\
\\
$(\rmnum{1})$ Consider any row of $M$ which represents the vertex $v$. Since the interchange of any two columns or rows preserves the totally unimodularity, then w.l.o.g, one may assume that it is the first row and its first coordinate $a_{11}$ equals to $1$. Then let $a^{\prime}_{i1}:=a_{i1}, \forall i \in \{1,\dots,m\}$ and calculate $a_{ij}$ with any fixed $j \geq 2$ as follows:
\begin{itemize}
\item [(a)] if $a_{1j} \neq 0$, do nothing; 
\item [(b)] if $a_{1j} = 0$, then let $b_{ij}= a_{ij}-a_{i1}$ and $c_{ij}= a_{ij}+a_{i1}$, for each $i \in \{1,\dots,m\}$;
\item [(c)] if $\vert b_{ij} \vert \leq 1$, then let $a^{\prime}_{ij}:=b_{ij}$. Otherwise, let $a^{\prime}_{ij}:=c_{ij}$.
\end{itemize}
For any fixed $j\geq 2$, we claim that $\vert a^{\prime}_{ij} \vert \leq 1$ for any $i$. Since all entries of totally unimodular matrix $M$ are in $\{-1,0,1\}$, the values of these new entries $a^{\prime}_{ij}$ may only be $\pm 1$, $\pm 2$ or $0$. Assume that the claim is false. Then we may assume that $a'_{i_1j}:=b_{i_1{j}}= -2$ and $a'_{i_2j}:=c_{i_2{j}}=2$, for some $i_1, i_2 \in \{1,\dots,m\}$. So in $M$, the relevant entries are $a_{i_1{1}}=1, \ a_{i_1{j}}=-1$ and $a_{i_2{1}}=1, \ a_{i_2{j}}=1$, which form a $2 \times 2$-submatrix
$$
\begin{pmatrix}
1 & -1 \\
1 & 1
\end{pmatrix}
$$
This obviously contradicts to the totally unimodularity.
\newline
\\
By computing all $j \geq 2$, we obtain a new matrix $M^{\prime}:=(a^{\prime}_{ij})$. Multiply $-1$ to the columns whose first entry is equal to $-1$ and denote the new matrix again as $M^{\prime}$. Then the first row of $M^{\prime}$ becomes $(1,\dots,1)$.
\\
Let $g$ be the linear map whose matrix $T$ represents the elementary operations above. Then it is unimodular as a multiplication of unimodular matrices. Denote the new polytope of the matrix $M^{\prime}$ as $P^{\prime}$. Now, $g$ is indeed such a unimodular transformation that for any vertex $v$ of $P$, we have $g(P)=P^{\prime}$, $M \cdot T = M^{\prime}$ and $g(v)=\sum_{t=1}^n {e_t}$. So $P$ satisfies the strong Ewald condition.
\\
\\
$(\rmnum{2})$ Let $f$ be a given face of $P^*$ with codimension $d$. Since $P$ is simplicial, $P^*$ is simple and then there are precisely $d$ facets s.t. $f=\bigcap_{i=1}^d F_i$, where $f\subset F_i \in \mathcal{F}(P')$ and $\mathcal{F}(P')$ is the set of all facets of $P^*$. Then it's obvious that:
\begin{itemize}
\item[(1)] $\mathrm{Star}(f)=\bigcup_{i=1}^d F_i$
\item[(2)] $\mathrm{star}(f)=\bigcup_{1 \leq i \neq j \leq d} (F_i \cap F_j)$
\end{itemize}
Let $\lambda$ be a symmetric point of $P^*$, then $-\lambda,\lambda \in P^*$ implies that $-1\leq \langle \lambda, u_i \rangle \leq 1$, where $u_i$ are normals of all the facets of $P^*$. Besides, these $u_i$ are vertices of $P$ and their coordinates are rows of $M$. In other words, we have $M \cdot \lambda^T \in {\{-1,0,1\}}^m$.
\\
\\
Now, our target is to find a specific $\lambda \in \mathcal{E}(P)$ satisfying the following requirements:
\begin{itemize}
\item[(1)] $\lambda \in \mathrm{Star}^{*}(f)$, i.e., among those facets $F_i$ containing $f$, the point $\lambda$ lies in exactly one of them. 
So, w.l.o.g., let $\lambda$ lie in $F_1$, then
$$\langle \lambda,u_1 \rangle =-1$$
and 
$$\langle \lambda,u_i \rangle \geq 0$$
where $2 \leq i \leq d$.
\item[(2)] $-\lambda \notin \mathrm{Star}(f)$, i.e., for all facets $F_i$ containing $f$, $-\lambda$ lies in none of them. So $\langle -\lambda,u_i \rangle \geq 0$, for all $i\in \{1,\dots,d\}$, i.e.
$$\langle \lambda,u_i \rangle \leq 0$$

\end{itemize}
As the consequence of (1) and (2), we obtain that
$$\langle \lambda,u_1 \rangle = -1, \ \langle \lambda,u_i \rangle =0, \ \forall i \in \{2,\dots,d\}$$
which can also be written as $U\cdot \lambda^T = -e_1^T$, where $U$ consists of rows $u_i$, and $e_1$ is the first unit basis vector of dimension $d$.
\\
\\
Since $P^*$ is simple, we can pick one vertex $v$ of $f$ and then $v=f\cap(\cap F_j)$, for some other facets $F_j \neq F_i$ with $j\in \{d+1,\dots,n\}$. Take the corresponding normals $u_i$, $i\in \{d+1,\dots,n\}$ and denote the matrix consists of $U$ and rows $u_i$, $i\in \{d+1,\dots,n\}$ as $N$. By the constructions, we know $u_i$, $i\in \{1,\dots,d\}$ are vertices of a facet and $P$ is smooth. Thus, the matrix $N$ is unimodular, which means $N^{-1}$ does exist. Moreover, by Lemma 1, $M \cdot N^{-1}$ is totally unimodular. Let $s_i=u_i \cdot \lambda^T, \forall i \in \{d+1,\dots,n\}$, then:
$$
N\cdot \lambda^T =
\begin{pmatrix}
-1 \\
0 \\
\vdots \\
0 \\
s_{d+1} \\
\vdots \\
s_n
\end{pmatrix}
\Longrightarrow
\lambda^T = N^{-1} \cdot
\begin{pmatrix}
-1 \\
0 \\
\vdots \\
0 \\
s_{d+1} \\
\vdots \\
s_n
\end{pmatrix}
$$  
Notice that $M\cdot \lambda^T \in {\{-1,0,1\}}^m$. So we must have $s_i \in \{-1,0,1\},\forall i\in\{d+1,\dots,n\}$. Then finding a specific $\lambda$ is the same as finding the concrete $s_i$ s,t.
$$
M\cdot \lambda^T = M\cdot N^{-1} \cdot 
\begin{pmatrix}
-1 \\
0 \\
\vdots \\
0 \\
s_{d+1} \\
\vdots \\
s_n
\end{pmatrix}
\in {\{-1,0,1\}}^m
$$
Now, we consider the collection of columns $\{1,d+1,\dots,n\}$ of $M\cdot N^{-1}$ and name these columns as $c_1,c_{d+1},\dots,c_n$, respectively. By Lemma $2$, we may find coefficients $a_1,a_{d+1},\dots,a_n \in \{-1,1\}$ with $a_1c_1 + a_{d+1}c_{d+1} + \cdots +a_nc_n \in {\{-1,0,1\}}^m$. If $a_0=-1$, then let $s_i=a_i$, otherwise, $s_i=-a_i$, for all $i \in \{d+1,\dots,n\}$. Then $\lambda^T = N^{-1} \cdot {(-1,0,\dots,0,s_{d+1},\dots,s_n)}^T$ is exactly the symmetric point we want. One can easily check that this $\lambda$ satisfies 
\begin{equation}
\left\{
\begin{aligned}
M\cdot \lambda^T & \in {\{-1,0,1\}}^m \\
U\cdot \lambda^T & = -e_1^T
\end{aligned}
\right.
\end{equation}
Therefore, $f$ satisfies the star Ewald condition. Since $f$ is an arbitrary face of the polytope $P^*$, it indeed satisfies the star Ewald condition.
\end{proof}

\subsection{Regular matroids and Seymour's decomposition theorem}

\begin{Def}
A matroid $D$ \emph{is an ordered pair $(E,\mathcal{I})$ consisting of a finite set $E$ and a collection $\mathcal{I}$ of subsets of $E$ satisfying the following three conditions:}
\begin{itemize}
\item[(C1)] $\emptyset \in \mathcal{I}$
\item[(C2)] If $I \in \mathcal{I}$ and $I' \subset I$, then $I' \in \mathcal{I}$
\item[(C3)] If $I_1$ and $I_2$ are in $\mathcal{I}$ and $\vert I_1 \vert < \vert I_2 \vert$, then there is an element $e$ of $I_2 - I_1$ such that $I_1 \cup e \in \mathcal{I}$
\end{itemize}
\end{Def}

If $D$ is the matroid $(E,\mathcal{I})$, then $D$ is called a matroid on $E$ with \textit{independent set} $\mathcal{I}$ and \textit{ground set} $E$. The \textit{rank} of the matroid $D$ is defined by the cardinality of the maximal subset of $E$ in $\mathcal{I}$. $N$ is called a \emph{minor} of $D$, if $N$ is obtained from $D$ by a sequence of contractions or deletions (\cite{OX}, Section $3.1$) or both. Let $A$ be a matrix over a fixed field $\mathbb{F}$. Take the set of its rows as the ground set and the collection of possible subsets as the independent set. Then $D(A)$ can be seen as a matroid induced by a matrix. If $D$ is isomorphic to $D(A)$, for some matrix $A$ over $\mathbb{F}$, then $D$ is said to be \emph{representable} over $\mathbb{F}$. A matroid is called \emph{regular} if it is representable over any field. Then we have the lemma:
\begin{Lem}[\cite{OX}, Lemma $2.2.21$] 
A matroid $D$ of non-zero rank $n$ is regular if and only if $D$ can be represented over $\mathbb{R}$ by some totally unimodular matrix of the form ${[I_r\vert A]}^T$, which is called the standard representative matrix for $D$.
\end{Lem}
\begin{Rem}
Let $P$ be a USFP and $M$ be its coordinate matrix, then $M$ can be seen as a representative matrix of a regular matroid. Indeed, this is because each row of $M$ gives the coordinates of a vertex of $P$, while the rows in the matrix of a matroid represent the elements of its ground set. 
\end{Rem}
Anyway, for regular matroids, it could be more convenient to deal with its corresponding totally unimodular matrix. 

\begin{Def} (operations that preserve totally unimodularity)
\begin{itemize}
\item[1.]\emph{(Dualizing) Let ${[I_{n \times n} \vert A_{n \times r}]}^T$ be a standard representative matrix of a matroid $D$, then its} (matroid) dual \emph{is defined by a representative matrix ${[-A^T_{r \times n} \vert I_{r \times r}]}^T$. (Here the name `matroid' dual is used to distinguish the `polar' dual of a polytope);}
\item[2.]\textup{($k$-sum) Let $A_1, A_2$ be the matrices of two matroids $D_1, D_2$, respectively. Then the} $k$-sum \textit{($k=1,2,3$) of $A_1$ and $A_2$ can be defined as follows:}
\par 
\textbf{1-sum}: 
$$
A_1 \oplus_1 A_2=
\begin{pmatrix}
A_1 & 0 \\
0 & A_2
\end{pmatrix} 
$$
\par
\textbf{2-sum}:
$$
\begin{pmatrix}
A_1 & a_1 \\
0 & 1
\end{pmatrix}
\oplus_2
\begin{pmatrix}
1 & 0 \\
a_2 & A_2
\end{pmatrix}
=
\begin{pmatrix}
A_1 & a_1 & 0\\
0 & a_2 & A_2\\
\end{pmatrix}
$$
\par
\textbf{3-sum}:
$$
\begin{pmatrix}
A_1 & a_1 & b_1 \\
0 & 1 & 1 \\
0 & 1 & 0 \\
0 & 0 &1
\end{pmatrix}
\oplus_3
\begin{pmatrix}
1 & 1 & 0 \\
1 & 0 & 0 \\
0 & 1 & 0 \\
a_2 & b_2 & A_2 
\end{pmatrix}
=
\begin{pmatrix}
A_1 & a_1 & b_1 & 0\\
0 & a_2 & b_2 & A_2
\end{pmatrix}
$$
\end{itemize}
\textup{where $a_1,a_2,b_1,b_2$ are column vectors and these matrices in the $k$-sums are all totally unimodular.}
\end{Def}
For more details and interesting topics in matroid, we refer the readers to books \cite{OX} by J. Oxley and \cite{Tru} by K. Truemper.

\begin{Def}
\textup{A matrix is called \emph{graphic} if it (or its transpose, respectively) contains at least one and at most two non-zero entries in each column (or row, respectively) and once there are precisely two non-zero entries, then one is $1$ and the other is $-1$. The (matroid) dual of a graphic matrix is said to be \emph{cographic}. 
In addition, let $G=(V,E)$ be an undirected graph. Then the graphic matroid is defined to be a matroid $D=(E,\mathcal{I})$, where $\mathcal{I}=\{F\subset E \ \vert \ F \ is \ acyclic \}$, i.e. forests in $G$} \emph{\cite{Tru}.} 
\end{Def}

\begin{Rem}
A matrix is said to be \emph{planar} if it's both graphic and cographic. Otherwise, it's called non-planar. The $k$-sum ($k=1,2,3$) of two graphic (cographic, respectively) matrix is again graphic (cographic, respectively). But the $k$-sum of a graphic matrix and a cographic matrix is neither graphic nor cographic if they are both not planar.
\end{Rem}

The following theorem is a very famous result for the charaterisation of regular matroids:

\begin{The} [Seymour's decomposition theorem, \cite{Seymour}]
Every regular matroid $D$ can be constructed by means of 1-, 2-, and 3-sums, starting with matroids each isomorphic to a minor of $D$ and each either graphic or cographic or isomorphic to $R_{10}$.
\end{The}

The matroid $R_{10}$ has representative matrix 
$$
m(R_{10}) = 
\begin{pmatrix}
1 & 0& 0& 0& 0 \\
0 & 1& 0& 0& 0\\
0 & 0& 1& 0& 0\\
0 & 0& 0& 1& 0\\
0 & 0& 0& 0& 1\\
-1& 1& 0& 0& 1\\
1& -1& 1& 0& 0\\
0& 1& -1& 1& 0\\
0& 0& 1& -1& 1\\
1& 0& 0& 1& -1\\
\end{pmatrix}
$$

In the language of matrices, it can be stated as:
\begin{The}
Every TU  matrix can be obtained from graphic matrices and $m(R_{10})$ by elementary operations of matrices, dualizing and $k$-sums for $k=1,2,3$.
\end{The}

Next, we are going to introduce another interesting type of smooth Fano polytopes, which arises from directed graphs \cite{Hig}.
\begin{Def}
\textup{Let $G=(V(G),A(G))$ be a directed graph and $e_1,\dots,e_d$ be the standard basis of $\mathbb{R}^d$. For each arrow $a=(i,j)$ of $G$, we define $\rho(a)\in \mathbb{R}^d$ by setting $\rho(a)=e_i-e_j$. Furthermore, the convex hull of $\{\rho(a)\ \vert \ e\in A(G)\}$ is denoted by $\mathcal{P}(G)$.}
\end{Def}
$\mathcal{P}(G)$ is obviously a polytope by definition.
\begin{Prop}[\cite{Hig}]
Let $\mathcal{P}(G)$ be a polytope arising from a directed graph $G$, then it is a Fano polytope of dimension $d-1$ if every arrow of $G$ lies in a directed cycle. 
\end{Prop}

If such a polytope is smooth, then we call it \emph{smooth Fano polytopes arising from digraph (SFPdG for short).} The property `smooth' has been characterised by A. Higashitani in (\cite{Hig}, Theorem $2.2$).

\subsection{Characterisation of unimodular smooth Fano polytopes}

With the help of Seymour's decomposition theorem and based on the property of the USFP, it's possible to induce a characterisation of USFPs.
\begin{Lem}[\cite{OX}, Theorem 13.3.2; \cite{Tru}, Theorem 10.4.2]
A regular matroid is graphic if and only if it has no minor isomorphic to $D^*(K_5)$ and $D^*(K_{3,3})$, where $D^*$ is the dual of the matroid $D$. In other words, a totally unimodular matrix is graphic if and only if it has no submatrix $m^{*}(K_5)$ or $m^{*}(K_{3,3})$, where $m^*$ is the dual of the matrix $m$.
\end{Lem}
Recall that $K_5$ is the complete graph with $5$ vertices and $K_{3,3}$ is the complete bipartite graph with $3$ vertices in each part.

\begin{Rem}
The matroids $D^*(K_{3,3})$ and $D^*(K_5)$ have representative matrices over $GF(3)$ as below, which are always denoted as $m^*(K_5)$ and $m^*(K_{3,3})$. 
$$
{m^*(K_{3,3})}=
\begin{pmatrix}
0& 1& -1& -1 \\
0& -1& 1& 0 \\
-1& 0& 1& 0 \\
-1& 1& 0& 1 \\
1& 0& 0& 0 \\
0& 1& 0& 0 \\
0& 0& 1& 0 \\
0& 0& 0& 1 
\end{pmatrix}
$$
and
$$
{m^*(K_{5})}=
\begin{pmatrix}
0& 0& 1& 0& -1& 1 \\
0& 1& 0& -1& 0& -1 \\
-1& 0& 0& 1& 1& 0 \\
1& -1& -1& 0& 0& 0 \\
1& 0& 0& 0& 0& 0 \\
0& 1& 0& 0& 0& 0 \\
0& 0& 1& 0& 0& 0 \\
0& 0& 0& 1& 0& 0 \\
0& 0& 0& 0& 1& 0 \\
0& 0& 0& 0& 0& 1 \\
\end{pmatrix}
$$
respectively. Generally speaking, the representative matrices are not necessarily unique over $GF(3)$.
\end{Rem}

Let $P$ be a USFP of dimension $n$ and $I$ be a subset of $\{1,\dots,n\}$. Then the $I$-projections of $P$ are defined to be the maps $f:P \longrightarrow P_I$, where $P_I$ is a polytope obtained by removing the $i$-th coordinate of $P$, for any $i \notin I$. When $x=0$, $\langle u_F,x \rangle = 0 >-1$ holds for any vector $u_F$, so the following lemma is obivously true. 
\begin{Lem}
The $I$-projections of $P$ preserve the projectivity of $P$.
\end{Lem}

\begin{The}[characterisation of USFP]
Let $P$ be a unimodular smooth Fano polytope. Then its coordinate matrix $M$ can be constructed from graphic matrices by elementary operations, dualizing and $k$-sums for $k=1,2,3$. 
\end{The}

\begin{proof} By the Seymour's decomposition theorem, $M$ can be decomposed into graphic matrices, cographic matrices and $m(R_{10})$ by elementary operations, dualizing and $k$-sums. Thus, we have to show that $m(R_{10})$ is not possible to appear as a submatrix. 
\par
Assume that $m(R_{10})$ appears in the decomposition of $M^T$, then it also appears as a submatrix of $M$ (up to elementary operations). By looking into the $k$-sums of the matrices, $R_{10}$ should be a part of the following form (w.l.o.g.,)
$$
M=
{\begin{pmatrix}
\cdots & A_1 & a_1 & b_1 & 0 & 0 & 0 & \cdots \\
\cdots & 0 & a & b & {m(R_{10})}  & c & 0 & \cdots \\
\cdots & 0 & 0 & 0 & 0       & c_2 & A_2 & \cdots
\end{pmatrix}}
$$
where $a_i,b_i,c_i$ are column vectors. Let $I$ be an index set containing all the columns of ${m(R_{10})}$, then the $I$-project maps $M$ to ${m(R_{10})}$. However, the sum of any row of ${m(R_{10})}$ equals to $1$. Then the polytope $P_{I}$ lies in the hyperplane $x_1 + x_2 +x_3 +x_4 +x_5 =1$ which does not contain the origin. This contradicts with the Lemma $5$.
\par
\end{proof}

\begin{Exa}
In dimension $4$, there is exactly one USFP which is not SFPdG since it possesses $m^*(K_{3,3})$ as a submatrix. It's called \textup{F.4D.0037} in the database polyDB \textup{\cite{Pol}} and its matrix is:
$$
\begin{pmatrix}
0&  1& -1& 1 \\
0&  0& -1& 0 \\
1& 1& -1& 0 \\
0& 1&  0& 1 \\
0& 0& 0& -1 \\
-1& 0& 0& 0 \\
-1& 1& 0& 0 \\
0& -1& 0& 0 \\
1& -1& 1& -1 
\end{pmatrix}
$$

Trivially, ${m^*(K_5)}$ is not a submatrix of a $5$-dimensional USFP because of the size. In addition, ${m^*(K_5)}$ should not be the coordinate matrix of a $6$-dimensional USFP, since the polytope with representative matrix
$$
{m^*(K_{5})}=
\begin{pmatrix}
0& 0& 1& 0& -1& 1 \\
0& 1& 0& -1& 0& -1 \\
-1& 0& 0& 1& 1& 0 \\
1& -1& -1& 0& 0& 0 \\
1& 0& 0& 0& 0& 0 \\
0& 1& 0& 0& 0& 0 \\
0& 0& 1& 0& 0& 0 \\
0& 0& 0& 1& 0& 0 \\
0& 0& 0& 0& 1& 0 \\
0& 0& 0& 0& 0& 1 \\
\end{pmatrix}
$$
is not simplicial. 
\end{Exa}

\begin{Exa}
Another couterexample in higher dimension could be: 
$$
\begin{pmatrix}
-1 & -1 & 0 & 0 & 0 & 0 \\
-1 & 0 & -1 & 0 & 0 & 1 \\
0 & -1 & 0 & -1 & -1 & -1 \\
0 & 0 & -1 & -1 & 0 & 0 \\
0 & 0 & 0 & 0 & 0 & 1 \\
0 & 0 & 0 & 0 & 1 & 0 \\
0 & 0 & 0 & 1 & 0 & 0 \\
0 & 0 & 1 & 0 & 0 & 0 \\
0 & 1 & 0 & 0 & 0 & 0 \\
1 & 0 & 0 & 0 & 0 & 0 \\
1 & 1 & 1 & 1 & 1 & 0 
\end{pmatrix}
$$
The totally unimodularity can be verified by checking its submatrix $S$ consisting of the first $4$ rows and the last row. Notice that the $2$nd, $3$rd and the last rows are linearly dependent. So any $5 \times 5$-submatrix of $S$ has determinant $0$. So, we may consider $S'$ by removing the last row of $S$. It is not difficult to find that $S'$ is the transpose of a graphic matrix up to multiplication of $-1$ on any row or column. The related graph has $5$ vertices and can be given by arrows:
$$
(1,2),(1,3),(4,2),(4,3),(5,3),(2,3).
$$
Thus, $S$ is totally unimodular and so is the original matrix. 
\\
Regarding facets, there can not be more than $6$ vertices lying on any given hyperplane. So the polytope above is also simplicial. By looking at the first four columns, one may easily check that it's unimodularly equivalent to $m^*(K_{3,3})$. Thus, the polytope above is a USFP but the matrix is definitely not graphic, i.e., it is not an SFPdG.
\end{Exa}

\section{Deeply Monotone Polytopes}

\subsection{Deeply Smoothness}
This subsection is mainly a brief introduction to Luis Crespo's works on deeply monotone polytopes.
\begin{Def}
\emph{Let $F$ be a face of codimension $k$ of a smooth lattice polytope $P$, obtained as the intersection of $k$ facets with facet inequalities $\langle u_i, x \rangle \leq b_i$, for $i=\{1,\dots,k\}$ (with $u_i$ primitive).} The first displacement \emph{of $F$ is defined to be:}
$$
F_0 = P \cap \{x \in \mathbb{R} \ |\ \langle u_i, x \rangle = b_i -1, \forall i=1,\dots,k \}.
$$
\end{Def}

\begin{Def}
\emph{A lattice polytope $P$ is said to be} UT-free \emph{if there is no $2$-face in $P$ that is a unimodular triangle.} A unimodular triangle \emph{is a lattice triangle whose edge contains no lattice point in its relative interior.}
\end{Def}

\begin{Lem}
[\cite{CPS}, Lemma 4.3] Let $P$ be a smooth lattice polytope and let $F$ be a facet of $P$. Let $F_0$ be the first displacement of $F$. Then:
\begin{itemize}
\item[0.] $F_0$ is a lattice polytope.
\item[1.] If $P$ is monotone (i.e. smooth reflexive), then $F_0$ is reflexive.
\item[2.] $F_0$ is normally isomorphic to $F$, except perhaps if $P$ has a $2$-face that is a unimodular triangle with an edge in $F$ and third vertex in $F_0$.
\end{itemize}
\end{Lem}

\begin{Def}
\emph{Let $v$ be a vertex of an $n$-dimensional lattice smooth polytope $P$, and let $u_1, \dots,u_n$ be the primitive edge vectors at $v$.} The corner parallelepiped of $P$ at $v$ \emph{is defined to be the parallelepiped:}
$$
\{v+\sum^{n}_{i=1} \lambda_i u_i \ |\ \lambda_i \in [0,1], \forall i \in \{1,\dots,n\}\}
$$
\emph{Then, $P$ is called} deeply smooth \emph{if it contains the corner parallelepipeds at all its vertices. We call $P$ a} deeply monotone polytope (DMP, for short) \emph{if it is deeply smooth and monotone.}
\end{Def}
By the definition, deeply smooth polytopes are alwalys UT-free.

\begin{The}
[\cite{CPS}] The following properties are equivalent for a smooth lattice polytope $P$:
\begin{itemize}
\item [1.] $P$ is deeply smooth.
\item [2.] The first displacement of every face $F$ of $P$ is normally isomorphic to $F$.
\item [3.] $P$ and the first displacement of all its faces are UT-free.
\end{itemize}
\end{The}

\begin{The}
[\cite{CPS}] Let $P$ be a DMP. If $u_1$ and $u_2$ are primitive edge vectors of $P$ at the same vertex $v$, then $\mathcal{E}(P)$ contains $u_1$, $u_2$, and at least one of $u_1 + u_2$ and $u_1 - u_2$. Furthermore, P satisfies the star Ewald condition and the strong Ewald condition.
\end{The}

\subsection{The relation between DMP and USFP}
To prove Theorem $7$ below, we need the following lemma: 
\begin{Lem}
Let $P$ be a UT-free monotone polytope, then for any vertex $v$ and the primitive edge vectors $u_1, \dots, u_d$ at $v$, the equality below always holds
$$v= - \sum_{i=1}^d u_i.$$
\end{Lem}

\begin{proof} Let $F_{1}$ be the facet of $P$ containing precisely $u_2, \dots, u_d$ and $F_{1d}$ be its first displacement. Since there is no unimodular triangle, $F_{1d}$ is normally isomorphic to $F_1$ by Lemma $6$. Thus, $v$ is moved from its original position in $F_1$ to a vertex $v_2$ of $F_{1d}$ along the edge containing $u_1$. Notice that $F_1$ has distance $1$ with $F_{1d}$. Then, $$v_2-v=u_1$$
Denote the primitive edge vectors of $F_{1d}$ moved from $u_i$ of $F_1$ again as $u_i$ and let $F_{2}$ be the facet of $F_{1d}$ (recall that it is a lattice polytope) containing precisely $u_3, \dots, u_d$. Then the first displacement of $F_2$, which is called $F_{2d}$, has a vertex $v_3$ moved from vertex $v_2$ in $F_{1d}$. By Theorem $5$, $F_{1d}$ is also reflexive, $F_{2}$ has distance $1$ with $F_{2d}$ and thus,
$$v_3-v_2=u_2$$
By induction, equations $v_i-v_{i-1}=u_{i-1}$ hold for $i \in \{2,\dots,d\}$. After taking the sum of them, we obtain that
$$v_d=v+\sum_{i=1}^{d-1} u_i$$
Obviously, $v_d = -u_d$ for the final 1-dimensional polytope $F_{d-1,d}$.

\par
\end{proof}

With the following consequence, we indeed extend Crespo's result \cite{CPS} to the case of USFPs.

\begin{The}
Any deeply monotone polytope is the dual polytope of a USFP.
\end{The}

\begin{proof} Let $P$ be a $d$-dimensional deeply monotone polytope and $P^*$ its dual polytope. For any vertex $v$ of $P$, denote $u_1, \dots, u_d$ as the primitive edge vectors at $v$. In particular, the set $\{ u_1, \dots, u_d \}$ is a basis of the lattice. Collecting all these primitive edge vectors for all vertices, we have $\{ u_1, \dots, u_m \}$ the set of all primitive edge vectors of $P$. 
\\
For any vertex $v$, we now consider the dual polyhedral cone of the cone generated by $u_1, \dots, u_d$. Let the edge vectors of the dual polyhedral cone be $u^{*}_1, \dots, u^{*}_d$ with $u^{*}_i (u_i) >0$, $u^{*}_i (u_j) =0$, where $1 \leq i \neq j \leq d$. By properties of reflexive polytopes, the dual cone is a cone over a facet of the dual polytope $P^*$ and vectors $u^{*}_1, \dots, u^{*}_d$ form a basis of the dual lattice. For the coordinate matrix $M^{*}$ of $P^{*}$, w.l.o.g, let the first $d$ rows be the vectors $u^{*}_1, \dots, u^{*}_d$, denote other rows as $u^{*}_{d+1}, \dots, u^{*}_m$.
\\
\\
By the previous lemma, $v=-u_1-\dots-u_d$. Since $P$ is deeply monotone, the parallepiped at $v$ is contained in the polytope. So we may have $v+u_1+\dots+u_{i-1}+u_{i+1}+\dots+u_d$ contained in the polytope, i.e., $-u_i$ lies in the polytope. Since $P^{*}$ is dual to $P$, we have $-u^{*}_i (u_k) \geq -1$, for any $k \in \{1,\dots,m\}$. Then with the condition $u^{*}_i (u_i) >0$ above, it yields that for any $i \in \{1,\dots,d\}$, $u^{*}_i (u_i) = 1$. Furthermore, because $u^{*}_i (u_j) = 0$ for each $i \neq j$, $\{u^{*}_i\ \vert \ i=1,\dots,d\}$ is a dual basis of $\{u_i\ \vert \ i=1,\dots,d\}$.
\\
\\
Since $u^{*}_1, \dots, u^{*}_d$ form a basis of the lattice, by a unimodualr transformation $T$, they can be transformed to the standard unit vectors $e^{*}_1, \dots, e^{*}_d$. After applying $T$ on the matrix $M^{*}$, we will obtain a new matrix $M'$, whose first $d$ rows are $e^{*}_1, \dots, e^{*}_d$ and others are again denoted as $u^{*}_{d+1}, \dots, u^{*}_m$. Meanwhile, $u_1,\dots,u_d$ will be transformed to the standard unit vectors $e_1,\dots,e_d$ in the dual lattice.
\\
\\
By $-1 \leq e_i (u^{*}_k) \leq 1$, ($\forall i\in \{1,\dots, d\}, \forall k \in \{1,\dots, m\}$), the entries of $u^{*}_k$ can only be $\pm 1$ or $0$.
\\
\\
Now, we assume that $M'$ is not totally unimodular. Then there exists a $t \times t$-submatrix ($t\leq d$) $A$ s.t. its determinant is an integer other than $-1,0,1$. Denote the entries of $A$ as $a_{ij}$, for $i,j \in \{1,\dots,t\}$. W.l.o.g, let $a_{11} \neq 0$.
For the first row, we wish to reduce all the entries except for $a_{11}$ to 0:
\begin{itemize}
\item[1.] if $a_{1j}=0$, do nothing. Otherwise,
\item[2.] if $a_{1j} = a_{11}$, then subtract the $j$-th column by the 1st column;
\item[3.] if $a_{1j} = - a_{11}$, then add the 1st column to the $j$-th column.
\end{itemize}
\textbf{Claim:} After doing these operations, the new matrix, which is denoted as $A_1$, has no entries other than $-1,0,1$.
\\
\textbf{Proof of the claim:} the entries can only be $-2,-1,0,1,2$. Assume that $-2$ or $2$ appears in the $j$-th column, then in the matrix N, there exists a $2 \times 2$-submatrix whose 1st column and 2nd column are obtained from the 1st column and $j$-th column of N, respectively, which has the form:
$$
\begin{pmatrix}
a & a \\
a & -a
\end{pmatrix}
or
\begin{pmatrix}
a & -a \\
a & a
\end{pmatrix}
$$
where $a$ is $1$ or $-1$.
However, they contradicts with the requirement $-1 \leq (e_1+e_j)(u^{*}_k) \leq 1$ or $-1 \leq (e_1-e_j)(u^{*}_k) \leq 1$ in Theorem $6$, for any $k \in \{1,\dots,m\}$.
\par
$\hfill\square$
\\
Then we consider the 2nd row and do the same as above. Since the determinant of the matrix is not 0, there must be a non-zero entry in the 2nd column. By rearranging the rows, we can make sure that $a_{22} \neq 0$.
After going through all the rows, the final new matrix which is denoted as $A_t$ is a lower-triangle matrix with $a_{ii}=\pm 1$, $i \in \{1,\dots,t\}$. Thus $det(A)=det(A_t) = a_{11} \cdots a_{tt} = \pm 1$, contradiction.
\par
\end{proof}

\begin{Rem}
With the results above, we have the following relations:
$$
\begin{matrix}
\{\textnormal{Deeply monotone polytopes}\} & \subset & \{\textnormal{Duals of USFPs}\} \\
\cap & & \cup \\
\{\textnormal{UT-free monotone polytopes}\} & & \{\textnormal{Duals of SFPdGs}\}
\end{matrix}
$$
\end{Rem}
\par
Now, we know that USFPs satisfy three Ewald conditions by Theorem $1$, and then, DMP and SFPdG follow from the inclusions.  Naturally, one may conjecture:

\begin{Con}
UT-free  monotone  polytopes are the duals of USFPs (up to unimodular equivalence). In particular, all the UT-free monotone polytopes satisfy the three star Ewald conditions.
\end{Con}
\begin{Rem}
Furthermore, we can replace the `USFPs' in Conjecture $2$ by the `SFPdGs', i.e., 
$$\{\textnormal{UT-free monotone polytopes}\} \subset \{\textnormal{Duals of SFPdGs}\}.$$
For example, the polytopes of Example $2$ and $3$ are both not SFPdGs and recall that the dual polytope of the former one is not UT-free. In Example $3$, the dual polytope has only six simplicial $2$-faces and one can directly observe unimodular triangles. 
\end{Rem}

\begin{Rem}
The numbers of these types of polytopes, up to dimension $5$ are listed below:
\\
\begin{table}[!h]
\renewcommand{\arraystretch}{1.5}
\label{table_example}
\centering
\begin{tabular}{|c|c|c|c|c|c|}
\hline
Dimension & SFPs & UT-free & DMPs & SFPdGs & USFPs \\
\hline
2 & 5 & 5 & 5 & 5 & 5\\
\hline 
3 & 18 & 16 & 16 & 16 & 16\\
\hline 
4 & 124 & 74 & 72 & 95 & 96 \\
\hline
5 & 866 & 336 & 300 & 551 & 554 \\
\hline 
\end{tabular}
\end{table}
\end{Rem}

\end{document}